\documentclass[12pt]{amsart}
\usepackage{amscd,amsmath,amsthm,amssymb,amsfonts,enumerate,hyperref}
\def\NZQ{\mathbb}               
\def\NN{{\NZQ N}}

\def\Gc{{\mathcal G}}

\def\Gc{{\mathcal G}}

\def\opn#1#2{\def#1{\operatorname{#2}}} 
	\opn\chara{char} \opn\length{\ell} \opn\pd{pd} \opn\rk{rk}
	\opn\projdim{proj\,dim} \opn\injdim{inj\,dim} \opn\rank{rank}
	\opn\depth{depth} \opn\grade{grade} \opn\height{height}
	\opn\embdim{emb\,dim} \opn\codim{codim}
	\opn\Cl{Cl}
	
	\opn\Tr{Tr} \opn\bigrank{big\,rank}
	\opn\superheight{superheight}\opn\lcm{lcm}
	\opn\trdeg{tr\,deg}
	\opn\rdeg{rdeg}
	\opn\reg{reg} \opn\lreg{lreg} \opn\ini{in} \opn\lpd{lpd}
	\opn\size{size} \opn\sdepth{sdepth}
	\opn\link{link}\opn\fdepth{fdepth}\opn\lex{lex}
	\opn\tr{tr}
	\opn\type{type}
	\opn\gap{gap}
	\opn\arithdeg{arith-deg}
	\opn\revlex{revlex}
	\opn\div{div} \opn\Div{Div} \opn\cl{cl} \opn\Cl{Cl}
	\opn\Spec{Spec} \opn\Supp{Supp} \opn\supp{supp} \opn\Sing{Sing}
	\opn\Ass{Ass} \opn\Min{Min}\opn\Mon{Mon}
	\opn\Ann{Ann} \opn\Rad{Rad} \opn\Soc{Soc}
	\opn\Im{Im} \opn\Ker{Ker} \opn\Coker{Coker} \opn\Am{Am}
	\opn\Hom{Hom} \opn\Tor{Tor} \opn\Ext{Ext} \opn\End{End}
	\opn\Aut{Aut} \opn\id{id}
	
	\opn\nat{nat}
	\opn\pff{pf}
	\opn\Pf{Pf} \opn\GL{GL} \opn\SL{SL} \opn\mod{mod} \opn\ord{ord}
	\opn\Gin{Gin} \opn\Hilb{Hilb}\opn\sort{sort}
	\opn\PF{PF}\opn\Ap{Ap}
	\opn\mult{mult}
	\opn\bight{bight}
	\opn\div{div}
	\opn\Div{Div}
	\opn\aff{aff}
	\opn\relint{relint} \opn\st{st}
	\opn\lk{lk} \opn\cn{cn} \opn\core{core} \opn\vol{vol}  \opn\inp{inp}
	\opn\nilpot{nilpot}
	\opn\link{link} \opn\star{star}\opn\lex{lex}\opn\set{set}
	\opn\width{wd}
	\opn\Fr{F}
	\opn\QF{QF}
	\opn\G{G}
	\opn\type{type}\opn\res{res}
	\opn\conv{conv}
	\opn\Int{Int}
	\opn\Deg{Deg}
	\opn\Sym{Sym}
	\opn\Con{Con}
	\opn\gr{gr}
	
	\def\pot#1#2{#1[\kern-0.28ex[#2]\kern-0.28ex]}

	\opn\dirlim{\underrightarrow{\lim}}
	\opn\inivlim{\underleftarrow{\lim}}
	\def\Implies{\ifmmode\Longrightarrow \else
		\unskip${}\Longrightarrow{}$\ignorespaces\fi}
	\def\implies{\ifmmode\Rightarrow \else
		\unskip${}\Rightarrow{}$\ignorespaces\fi}
	\def\iff{\ifmmode\Longleftrightarrow \else
		\unskip${}\Longleftrightarrow{}$\ignorespaces\fi}

	\let\:=\colon
	\newtheorem{Theorem}{Theorem}[section]
	\newtheorem{Lemma}[Theorem]{Lemma}
	\newtheorem{Corollary}[Theorem]{Corollary}
	\newtheorem{Proposition}[Theorem]{Proposition}
	\theoremstyle{definition}
	\newtheorem{Remark}[Theorem]{Remark}
	\newtheorem{Definition}[Theorem]{Definition}

    \makeatletter
\@namedef{subjclassname@2020}{%
  \textup{2020} Mathematics Subject Classification}
\makeatother

\begin{document}

\title[Bounded powers of edge ideals]{Bounded powers of edge ideals: regularity and linear quotients}

\author[T.~Hibi]{Takayuki Hibi}
\author[S.~A.~ Seyed Fakhari]{Seyed Amin Seyed Fakhari}

\address{(Takayuki Hibi) Department of Pure and Applied Mathematics, Graduate School of Information Science and Technology, Osaka University, Suita, Osaka 565--0871, Japan}
\email{hibi@math.sci.osaka-u.ac.jp}
\address{(Seyed Amin Seyed Fakhari) Departamento de Matem\'aticas, Universidad de los Andes, Bogot\'a, Colombia}
\email{s.seyedfakhari@uniandes.edu.co}



\begin{abstract}
Let $S=K[x_1, \ldots,x_n]$ denote the polynomial ring in $n$ variables over a field $K$ and let $I \subset S$ be a monomial ideal. For a vector $\mathfrak{c}\in\NN^n$, we set $I_{\mathfrak{c}}$ to be the ideal generated by monomials belonging to $I$ whose exponent vectors are componentwise bounded above by $\mathfrak{c}$. Also, let $\delta_{\mathfrak{c}}(I)$ be the largest integer $k$ such that $(I^k)_{\mathfrak{c}}\neq 0$. It is shown that for every graph $G$ with edge ideal $I(G)$, the ideal $(I(G)^{\delta_{\mathfrak{c}}(I)})_{\mathfrak{c}}$ is a polymatroidal ideal. Moreover, we show that for each integer $s=1, \ldots \delta_{\mathfrak{c}}(I(G))$, the Castelnuovo--Mumford regularity of $(I(G)^s)_{\mathfrak{c}}$ is bounded above by $\delta_{\mathfrak{c}}(I(G))+s$.
\end{abstract}

\maketitle

\section{Introduction} \label{sec1}
Let $S=K[x_1, \ldots,x_n]$ denote the polynomial ring in $n$ variables over a field $K$. Also, let $\NN$ denote the set of nonnegative integers. Consider a vector $\mathfrak{c}=(c_1, \ldots, c_n) \in \mathbb{N}^n$. A monomial $u=x_1^{a_1}\cdots x_n^{a_n}$ of $S$ is called {\it $\mathfrak{c}$-bounded} if $a_i\leq c_i$ for each $i=1, 2, \ldots, n$. In particular, $\mathfrak{c}$-bounded monomials with $\mathfrak{c}=(1,\ldots, 1)$ are squarefree monomials.

Let $I \subset S$ be a monomial ideal. For any vector $\mathfrak{c}\in \mathbb{N}^n$, the ideal generated by all $\mathfrak{c}$-bounded monomials belonging to $I$ will be denoted by $I_\mathfrak{c}$. Especially, if $I$ is generated by monomial of degree $d$ and $\mathfrak{c}=(d, \ldots, d)$, then $I_\mathfrak{c}=I$.  Furthermore, if $I = (x_1, \ldots, x_n)^d$, then $I_\mathfrak{c}$ is called an {\it ideal of Veronese type} (\cite[Example 12.2.8]{HHgtm260}). For any positive integer $s\geq 1$, the ideal $(I^s)_{\mathfrak{c}}$ is called the {\it $s$th $\mathfrak{c}$-bounded power of $I$} and is the main objective of this paper. If $\mathfrak{c}=(1,\ldots, 1)$, then $(I^s)_{\mathfrak{c}}$ is a squarefree monomial ideal and because of this reason, it is also called $s$th squarefree power of $I$. The $s$th squarefree power of $I$ is denoted by $I^{[s]}$.

Let $G$ be a simple graph with vertex set $V(G)=\{x_1, \ldots, x_n\}$ and edge set $E(G)$. The edge ideal of $G$ is defined to be the ideal$$I(G):=(x_ix_j : x_ix_j\in E(G))\subset S.$$

The study of squarefree powers of edge ideals was initiated in \cite{BHZ}. This study was continued by several authors (see e.g., \cite{cfl}, \cite{ef1}, \cite{ef}, \cite{ehhs}, \cite{ehhs1}, \cite{eh}, \cite{fhh}, \cite{fs}, \cite{s1}, \cite{s2}, \cite{s3}). In this paper, we study the $\mathfrak{c}$-bounded powers of edge ideal. Indeed, we are mostly interested in the study of Castelnuovo--Mumford regularity and the property of having linear quotients for these kind of ideals.

For a monomial ideal $I$, let $\delta_{\mathfrak{c}}(I)$ denote the highest integer $k$ such that $(I^k)_{\mathfrak{c}}\neq 0$. For example, when $I=I(G)$ is the edge ideal of a graph $G$ and $\mathfrak{c}=(1, \ldots, 1)$, then $\delta_{\mathfrak{c}}(I)$ is just the matching number ${\rm match}(G)$ of $G$. We prove in Proposition \ref{Boston} that if a monomial ideal $I$ has linear quotients, then for every vector $\mathfrak{c}$, the ideal $I_{\mathfrak{c}}$ has linear quotients too. As a consequence, if each component of $\mathfrak{c}$ is strictly positive, then $(I(G)^s)_{\mathfrak{c}}$ has linear quotients, for each integer $s=1, 2, \ldots, \delta_{c}(I(G))$ if and only if the complementary graph $\overline{G}$ is chordal, Theorem \ref{edge}.

We know from \cite[Theorem 1 on page 246]{w} that for every graph $G$, the ideal $I(G)^{[{\rm match}(G)]}$ is a matroidal ideal. As a generalization of this result, we prove in Theorem \ref{Essen} that for any vector $\mathfrak{c}\in \NN^n$, the ideal $(I(G)^{\delta_\mathfrak{c}(I(G))})_{\mathfrak{c}}$ is a polymatroidal ideal, and hence, it has linear quotients.

In Section \ref{sec5}, we investigate the Castelnuovo--Mumford regularity of bounded powers of edge ideals. As the main result of the section, we prove in Theorem \ref{regmain} that for each integer $s=1, \ldots \delta_{\mathfrak{c}}(I(G))$, the regularity of $(I(G)^s)_{\mathfrak{c}}$ is bounded above by $\delta_{\mathfrak{c}}(I(G))+s$. This generalizes the inequality ${\rm reg}\big(I(G)^{[s]}\big)\leq {\rm match}(G)+s$ which was proved in \cite{s3} for each integer $s=1, \ldots, {\rm match}(G)$.

\section{Preliminaries} \label{sec2}

In this section, we provide the definitions and basic facts which will be used in the next sections.

All graphs in this paper are simple, i.e., the graphs have no loops and no multiple edges. Let $G$ be a graph with vertex set $V(G)=\big\{x_1, \ldots,
x_n\big\}$ and edge set $E(G)$. We identify the vertices (resp. edges) of $G$ with variables (resp. corresponding quadratic monomials) of $S$. For a vertex $x_i$, the {\it neighbor set} of $x_i$ is $N_G(x_i)=\{x_j\mid x_ix_j\in E(G)\}$. We set $N_G[x_i]=N_G(x_i)\cup \{x_i\}$. For every subset $U\subset V(G)$, the graph $G\setminus U$ has vertex set $V(G\setminus U)=V(G)\setminus U$ and edge set $E(G\setminus U)=\{e\in E(G)\mid e\cap U=\emptyset\}$. A subgraph $H$ of $G$ is called {\it induced} provided that two vertices of $H$ are adjacent if and only if they are adjacent in $G$. For every graph $G$, its {\it complementary graph}
$\overline{G}$ is the graph with $V(\overline{G})=V(G)$ and two vertices $x_i$ and $x_j$ are adjacent in $\overline{G}$ if they are not adjacent in $G$. A graph $G$ is called {\it chordal} if it does not have a cycle of length at least four as an induced subgraph.

Let $G$ be a graph. A subset $M\subseteq E(G)$ is a {\it matching} if $e\cap e'=\emptyset$, for every pair of edges $e, e'\in M$. The cardinality of the largest matching of $G$ is called the {\it matching number} of $G$ and is denoted by ${\rm match}(G)$.

\begin{Definition} \cite[Definition 6.2]{b}
Let $G$ be a graph. Two vertices $u$ and $v$ ($u$ may be equal to $v$) are said to be {\it even-connected} with respect to an $s$-fold product $e_1 \cdots e_s$ of edges of $G$, if there is an integer $r\geq 1$ and a sequence $p_0, p_1, \ldots, p_{2r+1}$ of vertices of $G$ such that the following conditions hold.
\begin{itemize}
\item[(i)] $p_0=u$ and $p_{2r+1}=v$.

\item[(ii)] $p_0p_1, p_1p_2, \ldots, p_{2r}p_{2r+1}$ are edges of $G$.

\item[(iii)] For all $0\leq k\leq r-1, \{p_{2k+1},p_{2k+2}\}=e_i$ for some $i$.

\item[(iv)] For all $i$, $|\{k\mid \{p_{2k+1},p_{2k+2}\}=e_i\}| \leq |\{j\mid e_i=e_j\}|$.
\end{itemize}
Moreover, the sequence $p_0, p_1, \ldots, p_{2r+1}$ is called an {\it even-connection} between $u$ and $v$ with respect to $e_1 \cdots e_s$.
\end{Definition}

Let $M$ be a graded $S$-module. Suppose that the minimal graded free resolution of $M$ is given by
$$0\rightarrow \cdots \rightarrow \bigoplus_j S(-j)^{\beta _{1,j}(M)}\rightarrow \bigoplus_j S(-j)^{\beta _{0,j}(M)}\rightarrow M\rightarrow 0.$$
The Castelnuovo-Mumford regularity (or simply, regularity) of $M$, denoted by ${\rm reg}(M)$, is defined as
$${\rm reg}(M)={\rm max}\{j-i\mid \beta _{i,j}(M)\neq 0\}.$$Let $I$ be a homogeneous ideal generated in a single degree $d$. Then $I$ is said to have a linear resolution if ${\rm reg}(I)=d$.

For a monomial ideal $I$, the unique minimal set of its monomial generators will be denoted by $\Gc(I)$.  We say that $I$ has {\it linear quotients} if there is an ordering $u_1 \prec \cdots \prec u_m$ of the monomials belonging to $\Gc(I)$ for which the colon ideal $(u_1, \ldots, u_{i}):u_{i+1}$ is generated by a subset of $\{x_1, \ldots, x_n\}$ for each $1 \leq i < m$.  We then call $u_1 \prec \cdots \prec u_m$ a {\it linear quotients ordering} for $I$. We also say that $I$ has linear quotients with respect to the ordering $u_1 \prec \cdots \prec u_m$.  If $I$ is generated in a single degree and has linear quotients, then it has a linear resolution (\cite[Proposition 8.2.1]{HHgtm260}).

Let $I$ be a monomial ideal of $S$ which is
generated in a single degree. The ideal $I$ is called {\it polymatroidal} if
the following exchange condition is satisfied: For monomials $u=x^{a_1}_1
\ldots x^{a_n}_ n$ and $v = x^{b_1}_1\ldots x^{b_n}_ n$ belonging to $\Gc(I)$
and for every $i$ with $a_i > b_i$, there exists an integer $1\leq j\leq n$ with $a_j < b_j$ such that
$x_j(u/x_i)\in \Gc(I)$. A squarefree polymatroidal ideal is called {\it matroidal}. It is known that any polymatroidal ideal has linear quotients (see \cite[Theorem 12.6.2]{HHgtm260}).

Let $I$ be a monomial ideal of
$S$ with minimal generators $u_1,\ldots,u_m$,
where $u_j=\prod_{i=1}^{n}x_i^{a_{i,j}}$, $1\leq j\leq m$. For every $i$
with $1\leq i\leq n$, let $a_i=\max\{a_{i,j}\mid 1\leq j\leq m\}$, and
suppose that $$T=K[x_{1,1},x_{1,2},\ldots,x_{1,a_1},x_{2,1},
x_{2,2},\ldots,x_{2,a_2},\ldots,x_{n,1},x_{n,2},\ldots,x_{n,a_n}]$$ is a
polynomial ring over the field $K$. Let $I^{{\rm pol}}$ be the squarefree
monomial ideal of $T$ with minimal generators $u_1^{{\rm pol}},\ldots,u_m^{{\rm pol}}$, where
$u_j^{{\rm pol}}=\prod_{i=1}^{n}\prod_{k=1}^{a_{i,j}}x_{i,k}$, $1\leq j\leq m$. The ideal $I^{{\rm pol}}$
is called the {\it polarization} of $I$. We know from \cite[Corollary 1.6.3]{HHgtm260} that ${\rm reg}(I^{{\rm pol}})={\rm reg}(I)$.

\section{Monomial ideals with bounded powers and linear quotients} \label{sec3}
Let $S=K[x_1, \ldots,x_n]$ denote the polynomial ring in $n$ variables over a field $K$ and $I \subset S$ a monomial ideal. Suppose that $I$ has linear quotients and let $\mathfrak{c}\in \mathbb{N}^n$ be a vector. In the following proposition, we show that a linear quotients ordering of $I$ is inherited by $I_\mathfrak{c}$.

\begin{Proposition}
\label{Boston}
Let $I\subset S$ be a monomial ideal which has linear quotients. Then for every vector $\mathfrak{c}\in \NN^n$, the ideal $I_\mathfrak{c}$ has linear quotients.
\end{Proposition}

\begin{proof}
Assume that $\mathfrak{c}=(c_1, \ldots, c_n)$. Following the convention of dealing with linear quotients, for monomials $u$ and $v$ of $S$, we employ the notation
$$u:v=\frac{u}{{\rm gcd}(u,v)}.$$

Let $\Gc(I)=\{v_1, \ldots, v_s\}$ and suppose $I$ has linear quotients with respect to the ordering $v_1\prec \cdots \prec v_s$. Since $\Gc(I_{\mathfrak{c}})\subseteq \Gc(I)$, there exist $1 \leq \ell_1<\cdots<\ell_q \leq s$ with $\Gc(I_{\mathfrak{c}})=\{v_{\ell_1}, \ldots, v_{\ell_q}\}$. For each $1\leq k\leq q$, we set $u_k=v_{\ell_k}$. We claim that $I_\mathfrak{c}$ has linear quotients with respect to the ordering $u_1\prec \cdots \prec u_q$. Fix integers $i$ and $j$ with $1\leq i< j\leq q$. We must prove that there is $k$ with $1\leq k< j$ such that $u_k:u_j$ 
is a variable which divides $u_i:u_j$. 
Since $v_1\prec \cdots \prec v_s$ is a linear quotients ordering for $I$, there is $r < {\ell j}$ such that $v_r:v_{\ell_j}=v_r:u_j$ is a variable (say $x_p$) which divides $v_{\ell_i}:v_{\ell_j} = u_i:u_j$.  So, it is enough to prove that $v_r$ is a $\mathfrak{c}$-bounded monomial. Remind that $v_{\ell_i}$ and $v_{\ell_j}$ are $\mathfrak{c}$-bounded. Since $x_p$ divides $v_{\ell_i}:v_{\ell_j}$,
we conclude that $v_{\ell_j}$ is not divisible by $x_p^{c_p}$. Hence $x_pv_{\ell_j}$ is a $\mathfrak{c}$-bounded monomial. On the other hand, it follows from $v_r:v_{\ell_j} 
=x_p$ that $v_r$ divides $x_pv_{\ell_j}$.  Thus $v_r$ is a $\mathfrak{c}$-bounded monomial, as desired.
\end{proof}

Let $\mathfrak{c}=(c_1, \ldots, c_n)$ and $\mathfrak{c'}=(c'_1, \ldots, c'_n)$ belong to $\mathbb{N}^n$. We define $\mathfrak{c} \leq^\# \mathfrak{c'}$ if $c_i \leq c'_i$, for each integer $i=1, 2, \ldots, n$. As an immediate consequence of Proposition \ref{Boston}, we obtain the following corollary. 

\begin{Corollary}
\label{Istanbul}
Let $I \subset S$ be a monomial ideal and let $\mathfrak{c}$ and $\mathfrak{c'}$ be two vectors in $\mathbb{N}^n$ such that $\mathfrak{c'} \leq^\# \mathfrak{c}$.  If $I_\mathfrak{c}$ has linear quotients, then $I_\mathfrak{c'}$ has linear quotients too.
\end{Corollary}

\section{Edge ideals and linear quotients} \label{sec4}
In this section, we study the property of having linear quotients for bounded powers of edge ideals. The following theorem is a consequence of Fr\"oberg's theorem \cite[Theorem 9.2.3]{HHgtm260} and Proposition \ref{Boston}. It characterizes edge ideals with the property that all their bonded powers have linear quotients.

\begin{Theorem}
\label{edge}
Let $G$ be a finite graph and let $\mathfrak{c}=(c_1,\ldots,c_n)\in \mathbb{N}^n$ be a vector with each $c_i > 0$ for each $i=1, 2, \ldots, n$. Then the following conditions are equivalent:
\begin{itemize}
\item[(i)]
$(I(G)^s)_\mathfrak{c}$ has linear quotients for $s = 1,2,\ldots,\delta_{\mathfrak{c}}(I(G))$.
\item[(ii)]
$\overline{G}$ is chordal.
\end{itemize}
\end{Theorem}

\begin{proof}
(i) $\Rightarrow$ (ii) Since $(I(G))_\mathfrak{c} = I(G)$, the desired results follows from Fr\"oberg's theorem \cite[Theorem 9.2.3]{HHgtm260}.

\smallskip

(ii) $\Rightarrow$ (i) It follows from \cite[Theorem 10.1.9]{HHgtm260} that each power of $I(G)$ has linear quotients if $\overline{G}$ is chordal.  Hence, the desired results follows from Proposition \ref{Boston}.
\end{proof}

Recall that for every graph $G$ and any integer $s\geq 1$, the $s$th squarefree power of $I(G)$ is denoted by $I(G)^{[s]}$. The following corollary is a special case of Theorem \ref{edge}.

\begin{Corollary}
\label{squarefree}
For every graph $G$, the following conditions are equivalent:
\begin{itemize}
\item[(i)]
$I(G)^{[s]}$ has linear quotients for each integer $s = 1,2, \ldots, {\rm match}(G)$.
\item[(ii)]
$\overline{G}$ is chordal.
\end{itemize}
\end{Corollary}

We now study the highest non-vanishing bounded power of edge ideals. As mentioned in the Introduction, it is known by \cite[Theorem 1 on page 246]{w} that for every graph $G$, the ideal $I(G)^{[{\rm match}(G)]}$ is a matroidal ideal. As a generalization of this result, we prove in Theorem \ref{Essen} that $(I(G)^{\delta_{\mathfrak{c}}(I(G))})_\mathfrak{c}$ is a polymatroidal ideal. Recall that for a vector $\mathfrak{c}=(c_1, \ldots, c_n)$, its absolute value is $|\mathfrak{c}|=c_1+\cdots +c_n$. Also, for a monomial $u$ and a variable $x_i$, the degree of $u$ with respect to $x_i$ is denoted by ${\rm deg}_{x_i}(u)$.

\begin{Theorem}
\label{Essen}
For every vector $\mathfrak{c} \in \mathbb{N}^n$ and for any graph $G$, the ideal $(I(G)^{\delta_{\mathfrak{c}}(I(G))})_\mathfrak{c}$ is a polymatroidal ideal. In particular, it has linear quotients.
\end{Theorem}

\begin{proof}
Suppose that $\mathfrak{c}=(c_1, \ldots, c_n)$. To simplify the notation, set $I:=I(G)$ and $\delta:=\delta_{\mathfrak{c}}(I(G))$. We use induction on $|\mathfrak{c}|$. The assertion is trivial if $|\mathfrak{c}|\leq 2$. Therefore, assume that $|\mathfrak{c}|\geq 3$. Consider two monomials $u,v\in \Gc((I^{\delta})_{\mathfrak{c}})$ and suppose that ${\rm deg}_{x_i}(u)> {\rm deg}_{x_i}(v)$, for some integer $i$ with $1\leq i\leq n$. We must show that there is a variable $x_j$ with ${\rm deg}_{x_j}(u)< {\rm deg}_{x_j}(v)$ such that $x_j(u/x_i)\in \Gc((I^{\delta})_{\mathfrak{c}})$. Suppose that $u=e_1\cdots e_{\delta}$ and $v=f_1\cdots f_{\delta}$, where $e_1, \ldots, e_{\delta}$ and $f_1, \ldots, f_{\delta}$ are (not necessarily distinct) edges of $G$. First, assume that there are integers $p, q$ with $1\leq p,q\leq \delta$ such that $e_p=f_q$. Then the monomials $u':=u/e_p$ and $v':=v/e_p=v/f_q$ belong to $(I^{\delta-1})_{\mathfrak{c}'}$, where $\mathfrak{c}'=(c_1', \ldots, c_n')$ is the vector defined as follows.
$$c_k'=
\left\{
	\begin{array}{ll}
		c_k  & \mbox{if } x_k \ {\rm is \ not \ incident \ to} \ e_p=f_q \\
		c_k-1 & \mbox{if } x_k \ {\rm is \ incident \ to} \ e_p=f_q
	\end{array}
\right.$$

It is easy to see that $\delta_{\mathfrak{c}'}(I)=\delta-1$. Thus, we deduce from the induction hypothesis that there is a variable $x_j$ such that ${\rm deg}_{x_j}(u')< {\rm deg}_{x_j}(v')$ and $x_j(u'/x_i)\in (I^{\delta-1})_{\mathfrak{c}'}$. This implies that ${\rm deg}_{x_j}(u)< {\rm deg}_{x_j}(v)$ and $x_j(u/x_i)\in (I^{\delta})_{\mathfrak{c}}$. Therefore, we are done in this case. So, suppose that $e_p\neq f_q$, for any pair of integers $p,q$ with $1\leq p,q\leq \delta$.

Set $i_1:=i$. As $x_{i_1}=x_i$ divides $u$, we may assume that $e_1=x_{i_1}x_{j_1}$, for some vertex $x_{j_1}$ of $G$. In particular, $c_{j_1}\geq 1$. If $v$ is not divisible by $x_{j_1}$, then $ve_1\in (I^{\delta+1})_{\mathfrak{c}}$, contradicting the definition of $\delta$. Thus, $x_{j_1}$ divides $v$. So, we may assume that $f_1=x_{j_1}x_{i_2}$, for some vertex $x_{i_2}$ of $G$. If $x_{i_2}$ does not divide $u$, then $x_j:=x_{i_2}$ satisfies the desired condition. So, suppose that $x_{i_2}$ divides $u$. If $x_{i_1}=x_{i_2}$, then $e_1=f_1$, contradicting our assumption. Hence, $x_{i_2}\neq x_{i_1}$. Moreover, since $f_1=x_{j_1}x_{i_2}$ is an edge of $G$, we have $x_{i_2}\neq x_{j_1}$. Consequently, $x_{i_2}$ divides $u/e_1=e_2\cdots e_{\delta}$. Without loss of generality, we may assume that $e_2=x_{i_2}x_{j_2}$, for some vertex $x_{j_2}\neq x_{i_2}$ of $G$. If $x_{j_2}=x_{j_1}$, then $f_1=e_2$ which contradicts our assumption. So, $x_{j_2}\neq x_{j_1}$. If $x_{j_2}$ does not divide $v$, then$$x_{i_1}x_{j_2}v=e_1e_2f_2\cdots f_{\delta}\in (I^{\delta+1})_{\mathfrak{c}},$$contradicting the definition of $\delta$. Therefore, $x_{j_2}$ divides $v$. Since $x_{j_2}\neq x_{i_2}, x_{j_1}$, we conclude that $x_{j_2}$ divides $v/f_1=f_2\cdots f_{\delta}$. Hence, without loss of generality, we may assume that $f_2=x_{j_2}x_{i_3}$, for some vertex $x_{i_3}$ of $G$.

Continuing this process, we may suppose after $k$ steps (by relabeling the edges, if necessary) that for each integer $t=1, \ldots, k$, we have $e_t=x_{i_t}x_{j_t}$ and $f_t=x_{j_t}x_{i_{t+1}}$, where $x_{i_1}, \ldots, x_{i_k}$, as well as, $x_{j_1}, \ldots, x_{j_k}$ are distinct vertices of $G$. We divide the rest of the proof into the following steps.

\vspace{0.3cm}
{\bf Step 1.} Suppose that $x_{i_{k+1}}=x_{i_r}$, for some integer $r$ with $1\leq r\leq k$. Then$$e_r\cdots e_k=x_{i_r}\cdots x_{i_k}x_{j_r}\cdots x_{j_k}=f_r\cdots f_k.$$It follows that the monomials $u'':=u/(e_r\cdots e_k)$ and $v'':=v/f_r\cdots f_k$ belong to the ideal $(I^{\delta-(k-r+1)})_{\mathfrak{c-a}}$, where $\mathfrak{a}$ is the exponent vector of the monomial $e_r\cdots e_k=f_r\cdots f_k$. It is easy to see that $\delta_{\mathfrak{c-a}}(I)=\delta-(k-r+1)$. Thus, we deduce from the induction hypothesis that there is a variable $x_j$ such that ${\rm deg}_{x_j}(u'')< {\rm deg}_{x_j}(v'')$ and $x_j(u''/x_i)\in (I^{\delta-(k-r+1)})_{\mathfrak{c-a}}$. This implies that ${\rm deg}_{x_j}(u)< {\rm deg}_{x_j}(v)$ and $x_j(u/x_i)\in (I^{\delta})_{\mathfrak{c}}$. So, we are done. Therefore, we assume that $x_{i_{k+1}}\neq x_{i_r}$, for each integer $r=1, \ldots, k$.

\vspace{0.3cm}
{\bf Step 2.} Suppose that $x_{i_{k+1}}$ does not divide $u/(e_1\cdots e_k)$. Recall that by Step 1, $x_{i_{k+1}}$ is not equal to the vertices $x_{i_1}, \ldots, x_{i_k}$. Hence, if $x_{i_{k+1}}$ divides $e_{\ell}$, for some integer $\ell$ with $1\leq \ell\leq k$, then $x_{i_{k+1}}=x_{j_{\ell}}$ and thus, it also divides $f_{\ell}$. Moreover, as $f_k=x_{j_k}x_{i_{k+1}}$ is an edge of $G$, we have $x_{i_{k+1}}\neq x_{j_k}$. Therefore, $x_{i_{k+1}}$ does not divide $e_k$. This yields that
\begin{align*}
{\rm deg}_{x_{i_{k+1}}}(u) & ={\rm deg}_{x_{i_{k+1}}}(e_1\cdots e_k)={\rm deg}_{x_{i_{k+1}}}(e_1\cdots e_{k-1}) \leq {\rm deg}_{x_{i_{k+1}}}(f_1\cdots f_{k-1})\\ & < {\rm deg}_{x_{i_{k+1}}}(f_1\cdots f_k)\leq {\rm deg}_{x_{i_{k+1}}}(v).
\end{align*}
Set $x_j:=x_{i_{k+1}}$. By the above computation, we have ${\rm deg}_{x_j}(u)<{\rm deg}_{x_j}(v)$. On the other hand,$$x_ju/x_i=x_{i_{k+1}}u/x_{i_1}=f_1\cdots f_ke_{k+1}\cdots e_{\delta}\in (I^{\delta})_{\mathfrak{c}}.$$Consequently, $x_j$ satisfies the desired condition. Therefore, suppose that $x_{i_{k+1}}$ divides $u/(e_1\cdots e_k)=e_{k+1}\cdots e_{\delta}$. So, without loss of generality, we may suppose that $e_{k+1}=x_{i_{k+1}}x_{j_{k+1}}$, for some vertex $x_{j_{k+1}}$ of $G$.

\vspace{0.3cm}
{\bf Step 3.} Suppose that $x_{j_{k+1}}=x_{j_r}$, for some integer $r$ with $1\leq r\leq k$. Then$$e_{r+1}\cdots e_{k+1}=x_{i_{r+1}}\cdots x_{i_{k+1}}x_{j_{r+1}}\cdots x_{j_{k+1}}=f_r\cdots f_k.$$Thus, a similar argument as in Case 1 implies the assertion. So, we assume that $x_{j_{k+1}}\neq x_{j_r}$, for each integer $r=1, \ldots, k$.

\vspace{0.3cm}
{\bf Step 4.} Suppose that $x_{j_{k+1}}$ does not divide $v/(f_1\cdots f_k)$. As $e_{k+1}=x_{i_{k+1}}x_{j_{k+1}}$ is an edge of $G$, we have $x_{j_{k+1}}\neq x_{i_{k+1}}$. Moreover, note that$$e_1\cdots e_k=\frac{x_{i_1}(f_1\cdots f_k)}{x_{i_{k+1}}}.$$ Since $x_{j_{k+1}}$ does not divide $v/(f_1\cdots f_k)$ and $x_{j_{k+1}}\neq x_{i_{k+1}}$, we deduce from the above equality that
\begin{align*}
& {\rm deg}_{x_{j_{k+1}}}(v)={\rm deg}_{x_{j_{k+1}}}(f_1\cdots f_k)\leq {\rm deg}_{x_{j_{k+1}}}(e_1\cdots e_k) < \\ & {\rm deg}_{x_{j_{k+1}}}(e_1\cdots e_ke_{k+1})\leq {\rm deg}_{x_{j_{k+1}}}(u)\leq c_{j_{k+1}}.
\end{align*}
Furthermore, if $x_{j_{k+1}}=x_{i_1}$, in the above chain of inequalities, the first inequality will be strict. So, in this case, ${\rm deg}_{x_{j_{k+1}}}(v)< c_{j_{k+1}}-1$. In any case,$$x_{i_1}x_{j_{k+1}}v=e_1 \cdots e_{k+1}f_{k+1}\cdots f_{\delta}\in (I^{\delta+1})_{\mathfrak{c}},$$contradicting the definition of $\delta$. Hence, $x_{j_{k+1}}$ divides $v/(f_1\cdots f_k)=f_{k+1}\cdots f_{\delta}$. Without loss of generality, we may assume that $f_{k+1}=x_{j_{k+1}}x_{i_{k+2}}$, for some vertex $x_{i_{k+2}}$ of $G$. Therefore, the algorithm continues in this case. However, since ${\rm deg}(u)={\rm deg}(v)=\delta$ is a finite number, this process should stop. This means that in the last stage of the algorithm, only one of the Steps 1, 2 or 3 can occur, and this completes the proof.

The last part of the theorem follows from \cite[Theorem 12.6.2]{HHgtm260} and the first part.
\end{proof}

The following corollary is an immediate consequence of Theorem \ref{Essen}.

\begin{Corollary} \label{linres}
For every vector $\mathfrak{c} \in \mathbb{N}^n$ and for any graph $G$, the ideal $(I(G)^{\delta_{\mathfrak{c}}(I(G))})_\mathfrak{c}$ has a linear resolution.
\end{Corollary}

\begin{Remark}
In view of Theorem \ref{Essen}, one may expect that for any monomial ideal $I$ and any vector $\mathfrak{c}\in \NN^n$, the ideal $(I^{\delta_{\mathfrak{c}}(I)})_{\mathfrak{c}}$ has linear quotients. However, as is well-known, it is false in general. For example, let $I=(x_1x_2x_3, x_1x_4x_5)\subset S=K[x_1, \ldots,x_5]$ and set $\mathfrak{c}=(1,\dots,1)$. Then $\delta_{\mathfrak{c}}(I)=1$, but it is clear that $I$ does not have linear quotients.
\end{Remark}

\section{Edge ideals and regularity} \label{sec5}

In this section, we study the regularity of the bounded powers of edge ideals. Let $G$ be a graph and assume that $\mathfrak{c}$ is a vector in $\NN^n$. As the main result of this section, we prove in Theorem \ref{regmain} that for each integer $s=1, 2, \ldots, \delta_{\mathfrak{c}}(I(G))$, the regularity of $(I(G)^s)_{\mathfrak{c}}$ is bounded above by $\delta_{\mathfrak{c}}(I(G))+s$. The proof is a modification of the proof of \cite[Theorem 3.2]{s3}. A big difference is that here the ideals involved in the proof are not necessarily squarefree. For this reason, we apply the machinery of polarization.

We first need the following lemma, which determines a suitable ordering for the minimal monomial generators of bounded powers of edge ideals.

\begin{Proposition} \label{rfirst}
Assume that $G$ is a graph and $s\leq \delta_{\mathfrak{c}}(I(G))-1$ is a positive integer. Then the monomials in $\Gc((I(G)^s)_{\mathfrak{c}})$ can be labeled as $u_1, \ldots, u_m$ such that for every pair of integers $1\leq j< i\leq m$, one of the following conditions holds.
\begin{itemize}
\item [(i)] $(u_j:u_i) \subseteq ((I(G)^{s+1})_{\mathfrak{c}}:u_i)$; or
\item [(ii)] there exists an integer $r\leq i-1$ such that $(u_r:u_i)$ is generated by a variable, and $(u_j:u_i)\subseteq (u_r:u_i)$.
\end{itemize}
\end{Proposition}

\begin{proof}
Using \cite[Theorem 4.12]{b}, the elements of $\Gc(I(G)^s)$ can be labeled as $v_1, \ldots, v_t$ such that for every pair of integers $1\leq j< i\leq t$, one of the following conditions holds.
\begin{itemize}
\item [(1)] $(v_j:v_i) \subseteq (I(G)^{s+1}:v_i)$; or
\item [(2)] there exists an integer $k\leq i-1$ such that $(v_k:v_i)$ is generated by a variable, and $(v_j:v_i)\subseteq (v_k:v_i)$.
\end{itemize}

Since $\Gc((I(G)^s)_\mathfrak{c})\subseteq \Gc(I(G)^s)$, there exist positive integers $\ell_1, \ldots, \ell_m$ such that $\Gc((I(G)^s)_\mathfrak{c})=\{v_{\ell_1}, \ldots, v_{\ell_m}\}$. For every integer $k$ with $1\leq k\leq m$, set $u_k:=v_{\ell_k}$. We claim that this labeling satisfies the desired property. To prove the claim, we fix integers $i$ and $j$ with $1\leq j< i\leq m$. Based on properties (1) and (2) above, we divide the rest of the proof into two cases.

\vspace{0.3cm}
{\bf Case 1.} Assume that $(v_{\ell_j}:v_{\ell_i}) \subseteq (I(G)^{s+1}:v_{\ell_i})$. Recall that $v_{\ell_i}$ and $v_{\ell_j}$ are $\mathfrak{c}$-bounded monomials. So, one can easily check that $(v_{\ell_j}:v_{\ell_i})=(u)$, for some monomial $u$ with the property that $uv_{\ell_i}$ is a $\mathfrak{c}$-bounded monomial. Since $u\in (I(G)^{s+1}:v_{\ell_i})$, we conclude that $uv_{\ell_i}\in (I(G)^{s+1})_{\mathfrak{c}}$. Consequently,$$(u_j:u_i)=(v_{\ell_j}:v_{\ell_i})=(u)\subseteq ((I(G)^{s+1})_{\mathfrak{c}}: v_{\ell_i})=((I(G)^{s+1})_{\mathfrak{c}}: u_i).$$

\vspace{0.3cm}
{\bf Case 2.} Assume that there exists an integer $k\leq \ell_i-1$ such that $(v_k:v_{\ell_i})$ is generated by a variable, and $(v_{\ell_j}:v_{\ell_i})\subseteq (v_k:v_{\ell_i})$. Hence, $(v_k:v_{\ell_i})=(x_p)$, for some integer $p$ with $1\leq p\leq n$. It follows from the inclusion  $(v_{\ell_j}:v_{\ell_i})\subseteq (v_k:v_{\ell_i})$ that $x_p$ divides $v_{\ell_j}:v_{\ell_i}$. Consequently, $v_{\ell_i}x_p$ is a $\mathfrak{c}$-bounded monomial. As ${\rm deg}(v_k)={\rm deg}(v_{\ell_i})$, it follows from $(v_k:v_{\ell_i})=x_p$ that there is a variable $x_q$ dividing $v_{\ell_i}$ such that $v_k=x_pv_{\ell_i}/x_q$. This implies that $v_k$ is a $\mathfrak{c}$-bounded monomial. Hence, $v_k=v_{\ell_r}=u_r$, for some integer $r$ with $1\leq r\leq m$. Using $k\leq \ell_i-1$, we have $\ell_r\leq \ell_i-1$. Therefore, $r\leq i-1$ and$$(u_j:u_i)\subseteq(u_r:u_i)=(x_p).$$ This completes the proof.
\end{proof}

Using Proposition \ref{rfirst}, we obtain the following result, which provides a method to bound the regularity of $((I(G)^{s+1})_\mathfrak{c})$.

\begin{Theorem} \label{regcol}
Assume that $G$ is a graph and $s\leq \delta_{\mathfrak{c}}(I(G))-1$ is a positive integer. Let $\Gc((I(G)^s)_{\mathfrak{c}})=\{u_1, \ldots, u_m\}$ denote the set of minimal monomial generators of $(I(G)^s)_{\mathfrak{c}}$. Then$${\rm reg}((I(G)^{s+1})_\mathfrak{c})\leq \max\bigg\{\Big\{{\rm reg}\big((I(G)^{s+1})_\mathfrak{c}:u_i\big)+2s, 1\leq i\leq m\Big\}, {\rm reg}\big((I(G)^s)_\mathfrak{c}\big)\bigg\}.$$
\end{Theorem}

\begin{proof}
Without loss of generality, we may assume that the labeling $u_1, \ldots, u_m$ of elements of $\Gc((I(G)^s)_{\mathfrak{c}})$ satisfies conditions (i) and (ii) of Proposition \ref{rfirst}. This implies that for every integer $i\geq 2$,
\begin{align*}
\big(((I(G)^{s+1})_\mathfrak{c}, u_1, \ldots, u_{i-1}):u_i\big)=((I(G)^{s+1})_\mathfrak{c}:u_i)+({\rm some \ variables}).
\end{align*}
Hence, we conclude from \cite[Lemma 2.10]{b} that
\[
\begin{array}{rl}
{\rm reg}\big(((I(G)^{s+1})_\mathfrak{c}, u_1, \ldots, u_{i-1}):u_i\big) \leq {\rm reg}((I(G)^{s+1})_\mathfrak{c}:u_i).
\end{array} \tag{1} \label{1}
\]

For every integer $i$ with $0\leq i\leq m$, set $I_i:=((I(G)^{s+1})_\mathfrak{c}, u_1, \ldots, u_i)$. In particular, $I_0=(I(G)^{s+1})_\mathfrak{c}$ and $I_m=(I(G)^s)_\mathfrak{c}$. Consider the exact sequence
$$0\rightarrow S/(I_{i-1}:u_i)(-2s)\rightarrow S/I_{i-1}\rightarrow S/I_i\rightarrow 0,$$
for every $1\leq i\leq m$. It follows that$${\rm reg}(I_{i-1})\leq \max \big\{{\rm reg}(I_{i-1}:u_i)+2s, {\rm reg}(I_i)\big\}.$$Therefore,
\begin{align*}
& {\rm reg}((I(G)^{s+1})_{\mathfrak{c}})={\rm reg}(I_0)\leq \max\Big\{\{{\rm reg}(I_{i-1}:u_i)+2s, 1\leq i\leq m\}, {\rm reg}(I_m)\Big\}\\ & =\max\Big\{\big\{{\rm reg}(I_{i-1}:u_i)+2s, 1\leq i\leq m\big\}, {\rm reg}((I(G)^s)_{\mathfrak{c}})\Big\}.
\end{align*}
The assertion now follows from inequality (\ref{1}).
\end{proof}

Using Theorem \ref{regcol}, in order to bound the regularity of $(I(G)^{s+1})_{\mathfrak{c}}$, we need to study colon ideals of the form $((I(G)^{s+1})_{\mathfrak{c}}:u)$, where $u$ is a monomial in $\Gc((I(G)^s)_{\mathfrak{c}})$. In the following lemma, we show that these ideals are always quadratic monomial ideals.

\begin{Lemma} \label{deg2}
Assume that $G$ is a graph and $s\leq \delta_{\mathfrak{c}}(I(G))-1$ is a positive integer. Then for every monomial $u\in \Gc((I(G)^s)_{\mathfrak{c}})$, the ideal $((I(G)^{s+1})_{\mathfrak{c}}:u)$ is a monomial ideal generated in degree two.
\end{Lemma}

\begin{proof}
Let $w$ be a monomial in the set of minimal monomial generators of the ideal $((I(G)^{s+1})_{\mathfrak{c}}:u)$. In particular, $uw$ is a $\mathfrak{c}$-bounded monomial. Since$$w\in ((I(G)^{s+1})_{\mathfrak{c}}:u)\subseteq (I(G)^{s+1}:u),$$ it follows from \cite[Theorem 6.1]{b} that there is a quadratic monomial $v\in (I(G)^{s+1}:u)$ which divides $w$. Since $uv$ divides $uw$, we deduce that $uv$ is a $\mathfrak{c}$-bounded monomial and therefore, $v\in ((I(G)^{s+1})_{\mathfrak{c}}:u)$. Thus, we conclude from $w\in \Gc((I(G)^{s+1})_{\mathfrak{c}}:u)$ that $w=v$. Hence, $((I(G)^{s+1})_{\mathfrak{c}}:u)$ is a quadratic monomial ideal.
\end{proof}

The following corollary is a consequence of Lemma \ref{deg2} and determines the set of minimal monomial generators of the ideal $((I(G)^{s+1})_{\mathfrak{c}}:u)$.

\begin{Corollary} \label{graphh}
Let $G$ be a graph and $s\leq \delta_{\mathfrak{c}}(I(G))-1$ be a positive integer. Also, let $u=e_1\ldots e_s\in \Gc((I(G)^s)_{\mathfrak{c}})$ be a product of $s$ (not necessarily distinct) edges of $G$. Then $((I(G)^{s+1})_{\mathfrak{c}}:u)$ is generated by all quadratic monomials $x_ix_j$ (it is possible that $i=j$) such that $ux_ix_j$ is a $\mathfrak{c}$-bounded monomial and at least one the following conditions holds.
\begin{itemize}
\item[(i)] $x_i$ and $x_j$ are adjacent in $G$; or
\item[(ii)] $x_i$ and $x_j$ are even-connected in $G$ with respect to $e_1\ldots e_s$.
\end{itemize}
\end{Corollary}

\begin{proof}
By Lemma \ref{deg2}, the ideal $((I(G)^{s+1})_{\mathfrak{c}}:u)$ is a quadratic monomial ideal. Note that a monomial $x_ix_j$ belongs to $\Gc((I(G)^{s+1})_{\mathfrak{c}}:u)$ if and only if $ux_ix_j$ is a $\mathfrak{c}$-bounded monomial and $x_ix_j\in (I(G)^{s+1}:u)$. The assertion now follows from \cite[Theorems 6.5 and 6.7]{b}.
\end{proof}

The following lemma, is the main step in the proof of Theorem \ref{regmain}. It provides an upper bound for the regularity of the ideals in the form $((I(G)^s)_{\mathfrak{c}}:u)$, where $u$ is a minimal monomial generator of $(I(G)^{s-1})_{\mathfrak{c}}$.

\begin{Lemma} \label{colon}
Let $G$ be a graph and let $s$ be a positive integer with $2\leq s\leq \delta_{\mathfrak{c}}(I(G))$. Then for any monomial $u\in \Gc((I(G)^{s-1})_{\mathfrak{c}})$, we have$${\rm reg}\big((I(G)^s)_{\mathfrak{c}}:u\big)\leq \delta_{\mathfrak{c}}(I(G))-s+2.$$
\end{Lemma}

\begin{proof}
We prove the lemma by induction on $|\mathfrak{c}|$. Since, $\delta_{\mathfrak{c}}(I(G))\geq 2$, we must have $|\mathfrak{c}|\geq 4$. Suppose $|\mathfrak{c}|=4$. Then $s=2$ and the ideal $(I(G)^s)_{\mathfrak{c}}$ is a principal ideal. Hence, the assertion trivially holds in this case. So, suppose that $|\mathfrak{c}|\geq 5$. As $u$ is a minimal monomial generator of $(I(G)^{s-1})_{\mathfrak{c}}$, there are (not necessarily distinct) edges $e_1, \ldots, e_{s-1}$ of $G$ such that $u=e_1\cdots e_{s-1}$. We know from Lemma \ref{deg2} that $\big((I(G)^s)_{\mathfrak{c}}:u\big)$ is generated in degree two. Hence, there is a simple graph $H$ such that $I(H)=\big((I(G)^s)_{\mathfrak{c}}:u\big)^{\rm pol}$. By \cite[Corollary 1.6.3]{HHgtm260}, we have ${\rm reg}\big((I(G)^s)_{\mathfrak{c}}:u\big)={\rm reg}(I(H))$. Therefore, it is enough to prove that ${\rm reg}(I(H))\leq \delta_{\mathfrak{c}}(I(G))-s+2$. We have the following two cases.

\vspace{0.3cm}
{\bf Case 1.} Assume that every edge of $H$ is an edge of $G$ (in particular, the ideal $((I(G)^s)_{\mathfrak{c}}:u)$ is a squarefree monomial ideal). Let $M=\{f_1, \ldots, f_m\}$ be a matching of $H$ with $m={\rm match}(G)$. Note that the monomials $uf_1, \ldots, uf_m$ are $\mathfrak{c}$-bounded. Hence, their least common multiple $uf_1\cdots f_m$ is $\mathfrak{c}$-bounded too. This implies that $(I(G)^{s+m-1})_{\mathfrak{c}}\neq 0$. In particular, $m+s-1\leq \delta_{\mathfrak{c}}(I(G))$. Using \cite[Theorem 6.7]{hv}, we conclude that$${\rm reg}(I(H))\leq {\rm match}(H)+1\leq \delta_{\mathfrak{c}}(I(G))-s+2.$$

\vspace{0.3cm}
{\bf Case 2.} Assume that $H$ has an edge $xy$ which is not an edge of $G$. Set $z:=y$ if $y$ is a vertex of $G$ and set $z:=x$ if $y$ belongs to $V(H)\setminus V(G)$ (in this case $xy$ is obtained from polarizing the monomial $x^2\in ((I(G)^s)_{\mathfrak{c}}:u)$). By Corollary \ref{graphh}, the vertices $x$ and $z$ are even-connected in $G$ with respect to $e_1\ldots e_{s-1}$. For each $i=1, \ldots, s-1$, set $e_i=a_ib_i$ and let $P$ be an even-connection between $x$ and $z$ in $G$ with respect to $e_1\ldots e_{s-1}$. By relabeling $e_1, \ldots, e_{s-1}$ and relabeling $a_i, b_i$ (if necessary), we may assume that there is an integer $t$ with $1\leq t\leq s-1$ such that $P$ has the following form:$$P : x, a_1, b_1, a_2, b_2, \ldots, a_t, b_t, z.$$ In other words, the edges $e_1, \ldots, e_t$ appear in $P$ and the edges $e_{t+1}, \ldots, e_{s-1}$ are not involved in $P$. Set $v:=xze_1\cdots e_t$ and $w:=e_{t+1}\cdots e_{s-1}$ (if $t=s-1$, then $w=1$). Assume that $\mathfrak{a}$ is the exponent vector of $v$. Since $v$ is a $\mathfrak{c}$-bounded monomial, we have $\mathfrak{a} \leq^\# \mathfrak{c}$. We know from  Lemma \ref{deg2} that the ideal $((I(G)^{s-t})_{\mathfrak{c-a}}: w)$ is either zero or generated in degree two (it is zero if $\delta_{\mathfrak{c-a}}(I(G))< s-t$). So, there is a (possibly empty) simple graph $H'$ with $I(H')=((I(G)^{s-t})_{\mathfrak{c-a}}: w)^{\rm pol}$. By \cite[Corollary 1.6.3]{HHgtm260}, we have ${\rm reg}((I(G)^{s-t})_{\mathfrak{c-a}}: w)={\rm reg}(I(H'))$. Let $H''$ be the graph obtained from $H\setminus N_H[x]$ by deleting its isolated vertices.

\vspace{0.3cm}
{\bf Claim.} $H''$ is an induced subgraph of $H'$.

\vspace{0.3cm}
{\it Proof of the claim.} Assume that $x', y'$ are two vertices of $H''$ such that $x'y'\in E(H'')$. If $x'y'\in E(G)$, then clearly, $x'y'w\in I(G)^{s-t}$. Moreover, $x'y'\in E(H'')\subseteq E(H)$ implies that $x'y'u$ is a $\mathfrak{c}$-bounded monomial. On the other hand, it follows from $xz\in ((I(G)^s)_{\mathfrak{c}}:u)$ that $xzu$ is a $\mathfrak{c}$-bounded monomial. Since $x',y' \in V(H\setminus N_H[x])$, we have $x',y'\notin \{x,z\}$. Consequently, the least common multiple of the monomials $x'y'u$ and $xzu$ is $x'y'xzu=x'y'vw$. Since $x'y'u$ and $xzu$ are $\mathfrak{c}$-bounded monomials, we deduce that $x'y'vw$ is $\mathfrak{c}$-bounded, too. Consequently, $x'y'w$ is a ($\mathfrak{c-a}$)-bounded monomial. Therefore, $x'y'w\in (I(G)^{s-t})_{\mathfrak{c-a}}$. This implies that $x'y'\in E(H')$. Now, we consider the case that $x'y'\notin E(G)$. Without loss of generality, we may assume that $x'\in V(G)$. Set $z':=y'$ if $y'$ is a vertex of $G$ and set $z':=x'$ if $y'$ belongs to $V(H'')\setminus V(G)$. We first show that $x'z'\in (I(G)^{s-t}: w)$. Recall that $x',z'\notin N_H[x]$. Therefore, $x',z'\notin \{x,z\}$. Since $xzu$ and $x'z'u$ are $\mathfrak{c}$-bounded, their least common multiple, $xzx'z'u$ is also $\mathfrak{c}$-bounded. In particular, the monomials $xx'u$ and $xz'u$ are $\mathfrak{c}$-bounded.  On the other hand, we know that $x'$ and $z'$ are even-connected in $G$ with respect to $e_1\cdots e_{s-1}$. Let $Q$ be an even-connection between $x'$ and $z'$ in $G$ with respect to $e_1\cdots e_{s-1}$. If there is an integer $j$ with $1\leq j\leq t$ such that $e_j$ is appearing in $Q$, then we conclude from \cite[Lemma 6.13]{b} that either $xx'\in((I(G)^s)_{\mathfrak{c}}:u)$ or $xz'\in ((I(G)^s)_{\mathfrak{c}}:u)$. Thus, $x'\in N_H[x]$ or $z'\in N_H[x]$, which are both impossible. Therefore, the edges $e_1, \ldots, e_t$ do not appear in $Q$. It means that $x'$ and $z'$ are even-connected in $G$ with respect to $e_{t+1} \cdots e_{s-1}$. This yields that $x'z'\in (I(G)^{s-t}: w)$. Moreover, as mentioned above, the monomial $xzx'z'u=x'z'wv$ is $\mathfrak{c}$-bounded. Thus, $x'z'w$ is a ($\mathfrak{c}-\mathfrak{a}$)-bounded monomial. It follows that $x'z'\in ((I(G)^{s-t})_\mathfrak{c-a}: w)$, which means that $x'y'\in E(H')$. Therefore, we proved that $E(H'')\subseteq E(H')$ (this also shows that $V(H'')\subseteq V(H')$, as $H''$ does not have any isolated vertex). Therefore, $H''$ is a subgraph of $H'$.

To prove that this subgraph is induced, suppose $x'', y''$ are two vertices of $H''$ such that $x''y''\in E(H')$. Without loss of generality, we may assume that $x''\in V(G)$. Set $z'':=y''$ if $y''$ is a vertex of $G$ and set $z'':=x''$ if $y''$ belongs to $V(H'')\setminus V(G)$. By Corollary \ref{graphh}, either $x''z''\in E(G)$ or $x''$ and $z''$ are even-connected in $G$ with respect to $e_{t+1} \cdots e_{s-1}$. In both cases, $x''z''\in (I(G)^s:u)$. Moreover, since $x''z''w$ is ($\mathfrak{c}-\mathfrak{a}$)-bounded, we deduce that $x''z''u$ is a $\mathfrak{c}$-bounded monomial. So, $x''z''\in ((I(G)^s)_{\mathfrak{c}}:u)$. Equivalently,  $x''y''\in E(H'')$. Thus, $H''$ is an induced subgraph of $H'$ and this completes the proof of the claim.

\vspace{0.3cm}
It follows from the claim, together with \cite[Lemma 3.1]{ha} and the induction hypothesis that
\begin{align*}
{\rm reg}(I(H\setminus N_H[x]))& ={\rm reg}(I(H''))\leq {\rm reg}(I(H'))={\rm reg}((I(G)^{s-t})_{\mathfrak{c-a}}: w)\\ & \leq \delta_{\mathfrak{c}-\mathfrak{a}}(I(G))-(s-t)+2.
\end{align*}
(The last inequality is true even if the ideal $(I(G)^{s-t})_{\mathfrak{c-a}}$ is zero. Indeed, it follows from $w\in (I(G)^{s-t-1})_{\mathfrak{c-a}}$ that $\delta_{_{\mathfrak{c-a}}}(I(G))\geq s-t-1$ and hence, $\delta_{\mathfrak{c}-\mathfrak{a}}(I(G))-(s-t)+2\geq 1$.)  Note that for any monomial $u_0\in \Gc(I(G)^{\delta_{\mathfrak{c}-\mathfrak{a}}(I(G))})_{\mathfrak{c}-\mathfrak{a}})$, the monomial$$u_0v=u_0(xa_1)(b_1a_2)(b_2a_3) \cdots (b_{t-1}a_t)(b_tz)$$is a $\mathfrak{c}$-bounded monomial. Thus, $\delta_{\mathfrak{c}-\mathfrak{a}}(I(G))\leq \delta_{\mathfrak{c}}(I(G))-(t+1)$. Hence, we conclude from the above inequalities that
\[
\begin{array}{rl}
{\rm reg}(I(H\setminus N_H[x]))\leq \delta_{\mathfrak{c}}(I(G))-(t+1)-(s-t)+2=\delta_{\mathfrak{c}}(I(G))-s+1.
\end{array} \tag{2} \label{2}
\]

Now we study the regularity of $I(H\setminus x)$. Note that $I(H\setminus x)=((I(G)^s)_{\mathfrak{c'}}:u)^{\rm pol}$, where $\mathfrak{c'}$ is the vector obtained from $\mathfrak{c}$ by replacing the component corresponding to the vertex $x$ with ${\rm deg}_{x}(u)$. Since $x$ is a (non-isolated) vertex of $H$, we have $|\mathfrak{c'}|< |\mathfrak{c}|$. Thus, the induction hypothesis implies that
\[
\begin{array}{rl}
{\rm reg}(I(H\setminus x))\leq \delta_{\mathfrak{c}'}(I(G))-s+2\leq \delta_{\mathfrak{c}}(I(G))-s+2.
\end{array} \tag{3} \label{3}
\]

We know from \cite[Lemma 3.1]{dhs} that
\[
\begin{array}{rl}
{\rm reg}\big((I(G)^s)_{\mathfrak{c}}:u\big)={\rm reg}(I(H))\leq \max\big\{{\rm reg}(I(H\setminus N_H[x]))+1, {\rm reg}(I(H\setminus x))\big\}.
\end{array} \tag{4} \label{4}
\]
Therefore, the assertion follows from inequalities (\ref{2}), (\ref{3}) and (\ref{4}).
\end{proof}

We are now ready to prove the main result of this section.

\begin{Theorem} \label{regmain}
For any graph $G$ and for any positive integer $s\leq \delta_{\mathfrak{c}}(I(G))$, we have$${\rm reg}\big((I(G)^s)_{\mathfrak{c}}\big)\leq \delta_{\mathfrak{c}}(I(G))+s.$$
\end{Theorem}

\begin{proof}
Suppose that $\mathfrak{c}=(c_1, \ldots, c_n)$. If $c_i=0$, for some integer $i$ with $1\leq i\leq n$, then we replace $G$ by $G-x_i$. Therefore, we may assume that $c_i>0$, for each integer $i=1, 2, \ldots, n$.

We prove the theorem by induction on $s$. As $c_i\geq 1$ for each $i$, we deduce that ${\rm match}(G)\leq \delta_{\mathfrak{c}}(I(G))$. Hence, for $s=1$, the assertion follows from \cite[Theorem 6.7]{hv}. Thus, suppose $s\geq 2$. Let $\Gc((I(G)^{s-1})_{\mathfrak{c}})=\{u_1, \ldots, u_m\}$ denote the set of minimal monomial generators of $(I(G)^{s-1})_{\mathfrak{c}}$. Using the induction hypothesis, we have$${\rm reg}\big((I(G)^{s-1})_{\mathfrak{c}}\big)\leq \delta_{\mathfrak{c}}(I(G))+ (s-1)< \delta_{\mathfrak{c}}(I(G))+s.$$Moreover, Lemma \ref{colon} implies that for each integer $i=1, \ldots, m$, the inequality$${\rm reg}\big((I(G)^s)_{\mathfrak{c}}:u_i\big)\leq \delta_{\mathfrak{c}}(I(G))-s+2$$holds. Thus, the assertion follows from Theorem \ref{regcol}.
\end{proof}

\begin{Remark}
Note that for $s=\delta_{\mathfrak{c}}(I(G))$, the inequality of Theorem \ref{regmain} is indeed an equality. This, in particular, provides an alternative proof for Corollary \ref{linres}.
\end{Remark}

\end{document}